\documentclass[11pt]{article}
\usepackage[colorlinks,linkcolor=blue,citecolor=orange]{hyperref}

\usepackage[letterpaper]{geometry}
\usepackage{amsmath,amsthm,amsfonts,amssymb,booktabs}
\usepackage{enumerate,color,xcolor}
\usepackage{graphicx}
\usepackage{subfigure}
\usepackage{url}
\usepackage{todo}

\usepackage{amssymb}
\usepackage{amsthm}
\usepackage{amsmath}
\usepackage{mathrsfs}
\usepackage{bm}

\numberwithin{equation}{section}
\theoremstyle{plain}

\newtheorem{theorem}{Theorem}[section]
\newtheorem{lemma}{Lemma}[section]
\newtheorem{assumption}{Assumption}[section]

\newtheorem{corollary}[theorem]{Corollary}

\newtheorem{remark}[theorem]{Remark}

\usepackage{comment}

\newcommand{\beq}{\begin{eqnarray}}
\newcommand{\eeq}{\end{eqnarray}}
\newcommand{\beqs}{\begin{eqnarray*}}
\newcommand{\eeqs}{\end{eqnarray*}}

\makeatletter
\def\munderbar#1{\underline{\sbox\tw@{$#1$}\dp\tw@\z@\box\tw@}}
\makeatother

\begin{document}

\title{ Scaling limits for INAR$(\infty)$ processes}


\author{\bf Nian Yao \thanks{College of Mathematics and Statistics, Shenzhen University, Shenzhen, Guangdong Province, 518060, China. Email: yaonian@szu.edu.cn. ORCID: 0000-0002-9233-3384.}}

\date{}
\maketitle
\begin{abstract}
   In this paper, we study law of large numbers, central limit theorem, large and moderate deviations for INAR($\infty$) processes, which as a special case, includes both discrete-time linear Hawkes process and INAR(1) process in the literature. Our results recover existing results on large and moderate deviations for the discrete-time Hawkes process as studied in \cite{Wang2} and for the INAR(1) process as in \cite{Yu}.  
   \end{abstract}


\section{Introduction}
The theory and practice of statistical inference for integer-valued time series models are rapidly developing and they are important topics of the modern theory of statistics. One of the most successful integer-valued time series models proposed in the literature is the integer-valued autoregressive model of order $p$ (INAR($p$)). This model plays a crucial role in the study of statistical inference and was first introduced by \cite{AlzaidAl-Osh1990}. \cite{Gourieroux2004} introduces the INAR model with unobserved heterogeneity, applying it to premium updates in car insurance, and comparing it to the standard method based on the negative binomial distribution. \cite{Bisaglia2019} analyses several bootstrap techniques to conduct inference in INAR($p$) models. In an empirical application, they compare the performances of these methods in estimating the thinning parameters in INAR($p$) models. Then they employ model-based bootstrap methods to obtain coherent predictions and confidence intervals. Further applications of INAR processes can be found in various fields, including the relationship between personality factors and daily emotional experiences, syndromic surveillance, bivariate series concerning daytime and nighttime road accidents, public health, and real count time series. For a list of references regarding these applications, we refer to \cite{Bockenholt1999}, \cite{Andersson2010}, \cite{Pedeli2011}, \cite{Fontelo2016}, \cite{Mohammadpour2018}.

Many studies have focused on the asymptotic properties of the statistical law for the model.  \cite{AlzaidAl-Osh1988} examines the distribution and regression behavior of the INAR(1) process.\cite{ParkOh1997} demonstrates that the sample autocovariance function of the INAR(1) process with Poisson marginal distributions is asymptotically normally distributed and derived the asymptotic distribution of Yule-Walker-type estimators for the model parameters.\cite{ZhengBasawa2008} establishes the ergodicity of the INAR(1) process and obtained conditional least squares and maximum likelihood estimators of the model parameters. \cite{Silva2006} considers Yule-Walker parameter estimation for the INAR($p$) process. \cite{Bu2008} explores model selection,, estimation and forecasting for a class of INAR($p$) models suitable for use when analysing time series count data. \cite{Barczy2011} describes the asymptotic behavior of an unstable integer-valued autoregressive model of order $p$ (INAR($p$)). \cite{Nastic2012} introduces a new INAR($p$) model with geometric marginal distributions, denoted by CGINAR($p$), utilizing  negative binomial thinning. \cite{Pedeli2015} obtain conditional least squares estimation and maximum likelihood estimation  of high-order INAR($p$) processes, and propose a simple saddlepoint approximation to the log-likelihood that performs well even in the tails of the distribution and with complicated INAR models. \cite{LiJiangWang2018} show a large and moderate deviation result with explicit rate functions for a sequence of nearly critical INAR(1) processes.

When $p$ goes to $\infty$, the model might better capture the reality and be more relevant in practice, which leads to the study of INAR$(\infty)$ processes. An integer-valued autoregressive time series of infinite order (INAR$(\infty)$)
    is a sequence of random variables $(X_{n})_{n\in \mathbb{Z}}$
    such that
    \begin{equation}\label{main:dynamics}
    X_{n}=\varepsilon_{n}+\sum_{k=1}^{\infty}\sum_{l=1}^{X_{n-k}}\xi_{l}^{(n,k)},\qquad n\in\mathbb{Z},
\end{equation}
where $\varepsilon_{n}$ and $\xi_{l}^{(n,k)}$ are $\{0,1,2,\ldots\}$-valued random variables, and $\varepsilon_{n}$ are i.i.d. distributed as $\varepsilon$, and $\xi_{l}^{(n,k)}$ are i.i.d. over $n$ and $l$ distributed as $\xi_{k}$, and also independent of $(\varepsilon_{n})$. $\xi_{l}^{(n,k)}$ and $\xi_{l}^{(n,\tilde{k})}$
for $k, \tilde{k} \in \{1,2,\ldots,\}$, $k \neq \tilde{k}$ are independent.

For the special case when $\varepsilon_{n}$ are i.i.d. Poisson with mean $\alpha_{0}$,
and $\xi_{k}$ are Poisson with mean $\alpha_{k}$, then $\alpha_{0}$ is called
the immigration parameter, $(\varepsilon_{n})$ the immigration sequence, $\alpha_{k}\geq 0$
the reproduction coefficients, and $K:=\sum_{k=1}^{\infty}\alpha_{k}$ the reproduction mean.
In this special case, there is a strong link between INAR$(\infty)$ and the continuous-time
Hawkes process; see \cite{Kirchner}. Indeed, INAR$(\infty)$ includes the discrete-time Hawkes
process as a special example. \cite{Wang2, Wang1} study the limit theorems, large and moderate deviations for a discrete-time marked Hawkes process.
For the continuous-time linear Hawkes processes, \cite{Bordenave} obtained a large deviation principle and \cite{ZhuII} studied the moderate deviations.
The large deviation principle of a marked linear Hawkes process can be found in \cite{Karabash},
and the moderate deviation principle is studied in \cite{Seol3}. For nonlinear Hawkes processes, \cite{ZhuIII} established the level-3 large deviation principle. \cite{ZhuIV} proved the large deviations for Markovian Hawkes processes with exponential exciting function as well as a sum of exponentials functions as the exciting function.

We consider the INAR$(\infty)$ process starting with empty history, i.e.
$X_{n}=0$ for any $n\leq 0$ in \eqref{main:dynamics}, so that $X_{1}=\varepsilon_{1}$ and
for any $n\geq 2$,
\begin{equation}\label{main:dynamicsempty}
X_{n}=\varepsilon_{n}+\sum_{k=1}^{n-1}\sum_{l=1}^{X_{n-k}}\xi_{l}^{(n,k)}.
\end{equation}

Throughout the paper, we define $\Vert\alpha\Vert_{1}:=\sum_{i=1}^{\infty}\alpha_i$
as the $\ell_{1}$ norm of $\alpha$ where $\alpha$ is a nonnegative vector and assume that $\mathbb{E}[\varepsilon]<\infty$ and $\Vert\mathbb{E}[\boldsymbol{\xi}]\Vert_{1}<1$ where $\boldsymbol{\xi}:=(\xi_1,\xi_2,\ldots)$
and $\mathbb{E}[\boldsymbol{\xi}]:=(\mathbb{E}[\xi_1],\mathbb{E}[\xi_2],\ldots)$. 
We will show in Section~\ref{sec:LLN} that the following law of large numbers (LLN) holds:
\begin{equation}
\label{eq:lln}
\lim_{n\rightarrow\infty}\frac{1}{n}{\sum_{k=1}^{n}X_{k}}=\mu:=\frac{\mathbb{E}[\varepsilon]}{1-\Vert\mathbb{E}[\boldsymbol{\xi}]\Vert_{1}},
\end{equation}
almost surely as $n\rightarrow\infty$. It is thus natural to study the fluctuations around
the LLN limit $\mu$. In particular, we will derive the central limit theorem (CLT) in Section \ref{Central limit theorem}.

In \cite{Barczy2014} the asymptotic behavior of the conditional least squares estimators
of the autoregressive parameters, of the mean of the innovations, and of the stability parameter
for unstable INAR(2) processes is described. The large and moderate deviations for INAR(1) processes have been
studied in \cite{Yu}. In addition to the LLN and CLT, we are interested in deriving large and moderate deviations for INAR$(\infty)$ processes.
Even though there are already many studies of INAR$(\infty)$ processes and related models
in the literature, to the best of our knowledge, there are no studies on large and moderate deviations for INAR$(\infty)$ processes and our paper fills in this gap. The rest of the paper is organized as follows. In Section \ref{Main Results},
we state our main results. The large deviations for INAR$(\infty)$ processes are studied in Section \ref{Large deviations}, and the moderate deviations are provided in Section \ref{Moderate deviations}, followed by the law of large numbers and central limit theorems in Sections~\ref{sec:LLN}, \ref{Central limit theorem}. Finally, the proofs of the main results will be provided in Section \ref{Proof of Main Results}.

\section{Main Results}\label{Main Results}

Before we proceed to stating the main results of the paper,
let us first introduce large and moderate deviation principles.

\textbf{Large deviation principles.}
Following \cite{Dembo}, we introduce the definition of large deviation principle. A family of probability measures $\{\mathbf{P}_{n}\}_{n\in\mathbb{N}}$
on a topological space $\left(X,\mathcal{T}\right)$ satisfies the large deviation principle with rate function $I(\cdot):X\to[0,\infty]$ and speed $a_n$
if $I$ is a lower semi-continuous function, $(a_n)_{n\in\mathbb{N}}$ is a sequence of positive real numbers converging to $\infty$, and the following inequalities
hold for every Borel set $A$:
\begin{equation}
-\inf_{x\in A^{o}}I(x)\leq\liminf_{n\rightarrow\infty}\frac{1}{a_n}\log \mathbf{P}_{n}(A)
\leq\limsup_{n\rightarrow\infty}\frac{1}{a_n}\log \mathbf{P}_{n}(A)\leq-\inf_{x\in\overline{A}}I(x).
\end{equation}
Here $A^{o}$ is the interior of $A$ and $\overline{A}$ is the closure of A. We say that the rate function $I$ is good if for any $m\ge0$, the level set $\{x\in X:I(x)\le m, m\ge0\}$ is compact.
In addition to \cite{Dembo}, we also refer to \cite{VaradhanII} for a survey on large deviations.

\textbf{Moderate deviation principles.}
For any sequence $c_{n}$ such that $\sqrt{n}\ll c_n\ll n$, where $a_n\ll b_n$ means $a_n= o(b_n)$, as $n\rightarrow\infty$, a family of probability measures $\{\mathbf{P}_{n}\}_{n\in\mathbb{N}}$ on a topological space $\left(X,\mathcal{T}\right)$ satisfies a moderate deviation principle with rate function $J(\cdot):X\to[0,\infty]$
if $\{\mathbf{P}_{n}\}_{n\in\mathbb{N}}$ satisfies a large deviation principle with speed $\frac{c^2_n}{n}$ and with rate function $J(\cdot)$.

For example, let $X_1,\cdots,X_n$ be a sequence of i.i.d. random variables commonly distributed as $X$ and assume $\mathbb{E}\left[e^{\theta X}\right]<\infty$ for $\theta$ in some ball around the origin.
Then, $\mathbf{P}_{n}:=\mathbf{P}\left(\frac{1}{c_n}\sum^{n}_{i=1}X_i\in\cdot\right)$ satisfies a large deviation principle with speed $\frac{c^2_n}{n}$. Moderate deviations fills the gap between ordinary deviations approximated by the central limit theorem and large deviations.
We refer to \cite{Dembo} for the background on moderate deviation principles.


\begin{assumption}\label{CLT-Assumption}
\begin{itemize}
\item[(a)] $\Vert\mathbb{E}[\boldsymbol{\xi}]\Vert_{1}<1$ and 
$\Vert{\rm Var}[\boldsymbol{\xi}]\Vert_{1}<\infty$; 
\item[(b1)] $\sup\limits_{n>0}n^\frac{3}{2}\mathbb{E}[\xi_{n}]=C_{1}<\infty$;
\item[(b2)] $\sup\limits_{n>0}n^a\mathbb{E}[\xi_{n}]=C_{2}<\infty$ for some $a>\frac{3}{2}$; 
\item[(c)]$\mathbb{E}[\varepsilon]<\infty$ and ${\rm Var}[{\varepsilon}]<\infty$.

\end{itemize}
\end{assumption}

\subsection{Large deviations for INAR$(\infty)$ processes}\label{Large deviations}
For the  INAR$(\infty)$ process, we have the following large deviation principle.

\begin{theorem}\label{thm:LDP} Assume that Assumption \ref{CLT-Assumption} $(a),(c)$.
$\mathbf{P}(\frac{1}{n}\sum_{k=1}^{n}X_{k}\in\cdot)$
satisfies a large deviation principle with speed $n$ and rate function
\begin{equation}
I(x)=\sup_{\theta\leq\theta_{c}}\left\{\theta x-\log\mathbb{E}[e^{f_{\infty}(\theta)\varepsilon}]\right\},
\end{equation}
where $f_{\infty}(\theta)$ is the smaller solution
to the equation $F(x)=\theta$
for any $\theta\leq\theta_{c}:=\max_{x\in\mathbb{R}}F(x)$,
where
\begin{equation}
F(x):=x-\sum_{k=1}^{\infty}\log\mathbb{E}[e^{x\xi_{k}}].
\end{equation}
\end{theorem}

\begin{remark}
Because that $F(x)$ is a concave function in $x$, so the equation $F(x)=\theta$
may have two solutions, we choose the smaller one because we consider for any $\theta\leq \theta_c$ and $\theta\uparrow\theta_c$. 
\end{remark}

Let us consider two special cases. In the first special case, $\varepsilon$ is Poisson with mean $\alpha_{0}$,
and $\xi_{k}$ are Poisson with mean $\alpha_{k}$, $k=1,2,\ldots$.
This special case corresponds to the discrete-time linear Hawkes process
studied in \cite{Wang2}. Then, we have the following corollary.

\begin{corollary}\label{Coll1:LDP}
Assume $\varepsilon$ is Poisson with mean $\alpha_{0}$,
and $\xi_{k}$ are Poisson with mean $\alpha_{k}$, $k=1,2,\ldots$ and $\Vert\alpha\Vert_{1}<1$.
$\mathbf{P}(\frac{1}{n}\sum_{k=1}^{n}X_{k}\in\cdot)$
satisfies a large deviation principle with speed $n$ and the corresponding rate function
\begin{equation}
I(x)=\sup_{\theta\leq\theta_{c}}\left\{\theta x-\alpha_{0}(e^{f_{\infty}(\theta)}-1)\right\},
\end{equation}
where $f_{\infty}(\theta)$ satisfies the equation
\begin{equation}
f_{\infty}(\theta)=\theta+\Vert\alpha\Vert_{1}\left(e^{f_{\infty}(\theta)}-1\right),
\end{equation}
for any $\theta\leq\theta_{c}=\Vert\alpha\Vert_{1}-\ln\Vert\alpha\Vert_{1}-1$.
\end{corollary}

\begin{remark}
Our Corollary~\ref{Coll1:LDP} recovers the large
deviations for the discrete-time Hawkes process
that is studied in \cite{Wang2}.  
\end{remark}

In the second special case, we assume $\xi_{k}\equiv 0$ for every $k\geq 2$,
so that it corresponds to the INAR(1) model.

\begin{corollary}\label{Coll2:LDP}
Assume the INAR(1) model where $\xi_{k}\equiv 0$ for every $k\geq 2$ and $\mathbb E[\xi_1]<1$.
$\mathbf{P}(\frac{1}{n}\sum_{k=1}^{n}X_{k}\in\cdot)$
satisfies a large deviation principle with speed $n$ and the corresponding rate function
\begin{equation}
I(x)=\sup_{\psi\in\mathbb{R}}\left\{x\left(\psi-\log\mathbb{E}[e^{\psi\xi_{1}}]\right)-\log\mathbb{E}[e^{\psi\varepsilon}]\right\}.
\end{equation}
\end{corollary}

\begin{remark} Our Corollary \ref{Coll2:LDP} recovers the large deviations for the INAR(1) model studied in \cite{Yu}.	For the nearly unstable INAR(1) model studied in \cite{LiJiangWang2018}, our Corollary \ref{Coll2:LDP} also recovers the large deviations with the given scaling in their theorem 1.1.

\end{remark}

\subsection{Moderate deviations for INAR$(\infty)$ processes}\label{Moderate deviations}

Before we state the main result for the moderate deviations of the INAR$(\infty)$ processes, recall the equation (\ref{eq:lln}), where $\mu$ denotes the limit in the law of large numbers for $\frac{1}{n}\sum_{k=1}^{n}X_{k}$.
We have the following moderate deviation principles for the INAR$(\infty)$ processes.

\begin{theorem}\label{Thm:MDP}
Assume that Assumption \ref{CLT-Assumption} $(a)(b1)(c)$.
For any Borel set $A$ and time sequence $c(n)$ such that $\sqrt{n}\ll c(n)\ll n$, we have the following moderate deviation principle.
\begin{equation}\label{MDP}
\begin{aligned}
-\inf_{x\in A^{o}}J(x)
&\leq\liminf_{n\rightarrow\infty}\frac{n}{c^2(n)}\log \mathbf{P}\left(\frac{\sum_{k=1}^{n}X_{k}-n\mu}{c(n)}\in A\right)\\
&\leq\limsup_{n\rightarrow\infty}\frac{n}{c^2(n)}\log \mathbf{P}\left(\frac{\sum_{k=1}^{n}X_{k}-n\mu}{c(n)}\in A\right)\leq-\inf_{x\in\overline{A}}J(x),
\end{aligned}
\end{equation}
where \begin{equation}
    J(x)=\frac{x^2(1-\Vert\mathbb{E}[\boldsymbol{\xi}]\Vert_{1})^3}
{2(\mathbb{E}[\varepsilon]\Vert{\rm Var}[\boldsymbol{\xi}]\Vert_{1}+{\rm Var}[\varepsilon](1-\Vert\mathbb{E}[\boldsymbol{\xi}]\Vert_{1}))},
\end{equation}
and $\mu$ is given in (\ref{eq:lln}) and $\boldsymbol{\xi}:=(\xi_1,\xi_2,\ldots)$ and ${\rm Var}[\boldsymbol{\xi}]:=({\rm Var}[\xi_1],{\rm Var}[\xi_2],\ldots)$.
\end{theorem}

\begin{remark} 
We cannot have constant $\varepsilon$ and $\xi_k$ at the same time otherwise the model would be deterministic. But we can talk about what happens when one of them takes a constant.
\end{remark}

Next, let us consider the counterpart of Corollary~\ref{Coll1:LDP} and Corollary~\ref{Coll2:LDP} for the moderate deviations
for the special case of the discrete-time linear Hawkes process and the INAR(1) process.
Let us first consider the special case of the discrete-time linear Hawkes process, and we have the following result. 

\begin{corollary}\label{Coll1:MDP}
Assume that Assumption \ref{CLT-Assumption} $(a)(b1)(c)$. Assume $\varepsilon$ is Poisson with mean $\alpha_{0}(\neq0)$,
and $\xi_{k}$ are Poisson with mean $\alpha_{k}$, $k=1,2,\ldots$ and $\Vert\alpha\Vert_{1}<1$. Then for any Borel set $A$ and time sequence $c(n)$ such that $\sqrt{n}\ll c(n)\ll n$,
we have the moderate deviation principle (\ref{MDP}),
where \begin{equation}
    J(x)=\dfrac{x^2\left(1-\Vert \alpha\Vert_{1}\right)^3}{2\alpha_0}.
\end{equation}
\end{corollary}

\begin{remark} Our Corollary~\ref{Coll1:MDP} recovers the moderate deviations for the discrete-time Hawkes process
that is studied in \cite{Wang2}. \end{remark}

In the second special case, we consider INAR(1) model
where we assume $\xi_{k}\equiv 0$ for every $k\geq 2$.

\begin{corollary}\label{Coll2:MDP}
Assume that Assumption \ref{CLT-Assumption} $(a)(b1)(c)$. Assume the INAR(1) model where $\xi_{k}\equiv 0$ for every $k\geq 2$ and $\mathbb E[\xi_1]<1$.
Then for any Borel set $A$ and time sequence $c(n)$ such that $\sqrt{n}\ll c(n)\ll n$, we have the moderate deviation principle (\ref{MDP}),
where \begin{equation}
    J(x)=\dfrac{x^2\left(1-\mathbb{E}[\xi_1]\right)^3}{2\left(\mathbb{E}[\varepsilon]{\rm Var}[\xi_1]+{\rm Var}[\varepsilon](1-\mathbb{E}[\xi_1])\right)}.
\end{equation}
\end{corollary}

\begin{remark}
It is easy to see that our Corollary~\ref{Coll2:MDP} recovers the results of Theorem 1.2 in \cite{Yu}. But for the nearly unstable INAR(1) model studied in \cite{LiJiangWang2018}, our Corollary \ref{Coll2:MDP} can not recover the moderate deviations in their theorem 1.2.
\end{remark}


\subsection{Law of large numbers of INAR$(\infty)$ processes}\label{sec:LLN}

For the  INAR$(\infty)$ process, we have the following law of large numbers (LLN).

\begin{theorem}\label{thm:LLN} Assume that Assumption \ref{CLT-Assumption} $(a),(c)$.
\begin{equation}
\label{eq:lln:thm}
\lim_{n\rightarrow\infty}\frac{1}{n}{\sum_{k=1}^{n}X_{k}}=\mu:=\frac{\mathbb{E}[\varepsilon]}{1-\Vert\mathbb{E}[\boldsymbol{\xi}]\Vert_{1}},
\end{equation}
almost surely as $n\rightarrow\infty$.
\end{theorem}

Let us consider two special cases. In the first special case, $\varepsilon$ is Poisson with mean $\alpha_{0}$,
and $\xi_{k}$ are Poisson with mean $\alpha_{k}$, $k=1,2,\ldots$.
This special case corresponds to the discrete-time linear Hawkes process
studied in \cite{Wang1}. Then, we have the following corollary.

\begin{corollary}\label{Coll1:LLN}
Assume $\varepsilon$ is Poisson with mean $\alpha_{0}$,
and $\xi_{k}$ are Poisson with mean $\alpha_{k}$, $k=1,2,\ldots$ and $\Vert\alpha\Vert_{1}<1$.
\begin{equation}
\label{eq:lln:Discrete Hawkes}
\lim_{n\rightarrow\infty}\frac{1}{n}{\sum_{k=1}^{n}X_{k}}=\mu:=\frac{\alpha_{0}}{1-\Vert\alpha\Vert_{1}},
\end{equation}
almost surely as $n\rightarrow\infty$.

\end{corollary}

\begin{remark}
Our Corollary~\ref{Coll1:LLN} recovers the law of large number for the discrete-time Hawkes process[Theorem 2.1] that is studied in \cite{Wang1}.  
\end{remark}

In the second special case, we assume $\xi_{k}\equiv 0$ for every $k\geq 2$,
so that it corresponds to the INAR(1) model.

\begin{corollary}\label{Coll2:LLN}
Assume the INAR(1) model where $\xi_{k}\equiv 0$ for every $k\geq 2$. Assume that $\mathbb E[\xi_1]<1$ and $\mathbb{E}[\varepsilon]<\infty$.
\begin{equation}
\label{eq:lln:thm}
\lim_{n\rightarrow\infty}\frac{1}{n}{\sum_{k=1}^{n}X_{k}}=\mu:=\frac{\mathbb{E}[\varepsilon]}{1-\mathbb{E}[\xi_1]},
\end{equation}
almost surely as $n\rightarrow\infty$.
\end{corollary}

\begin{remark} Our Corollary \ref{Coll2:LLN} recovers the law of large numbers for the INAR(1) model[Lemma 1] studied in \cite{Pakes1971}.
\end{remark}

\subsection{Central limit theorem for INAR$(\infty)$ processes}\label{Central limit theorem}

For the  INAR$(\infty)$ process, we have the following central limit theorem (CLT).

\begin{theorem}\label{Thm:CLT} Under Assumption \ref{CLT-Assumption} $(a)(b2)(c)$.
\begin{equation}\label{eq:CLT}
\frac{\sum_{k=1}^{n}X_{k}-n\mu}{\sqrt{n}}\rightarrow N\left(0, \frac{\mathbb{E}[\varepsilon]\Vert{\rm Var}[\boldsymbol{\xi}]\Vert_{1}+{\rm Var}[\varepsilon](1-\Vert\mathbb{E}[\boldsymbol{\xi}]\Vert_{1}))}
{(1-\Vert\mathbb{E}[\boldsymbol{\xi}]\Vert_{1})^3} \right)
\end{equation}
in distribution as $n\rightarrow \infty$.
\end{theorem}

Next, let us consider the counterpart of Corollary~\ref{Coll1:LLN} and Corollary~\ref{Coll2:LLN} for the law of large numbers for the special case of the discrete-time linear Hawkes process and the INAR(1) process.
Let us first consider the special case of the discrete-time linear Hawkes process, and we have the following result. 

\begin{corollary}\label{Coll1:CLT}
Assume that Assumption \ref{CLT-Assumption} $(a)(b1)(c)$. Assume $\varepsilon$ is Poisson with mean $\alpha_{0}(\neq0)$, and $\xi_{k}$ are Poisson with mean $\alpha_{k}$, $k=1,2,\ldots$ and $\Vert\alpha\Vert_{1}<1$. 
\begin{equation}\label{eq:CLT1}
\frac{\sum_{k=1}^{n}X_{k}-n\mu}{\sqrt{n}}\rightarrow N\left(0, \frac{\alpha_0}{(1-\Vert\alpha\Vert_{1})^3} \right)
\end{equation}
in distribution as $n\rightarrow \infty$.
\end{corollary}

\begin{remark} Our Corollary~\ref{Coll1:CLT} recovers the central limit theorem for the discrete-time Hawkes process[Theorem 2.3] that is studied in \cite{Wang1}. \end{remark}

In the second special case, we consider INAR(1) model
where we assume $\xi_{k}\equiv 0$ for every $k\geq 2$.

\begin{corollary}\label{Coll2:CLT}
Assume that Assumption \ref{CLT-Assumption} $(a)(b1)(c)$. Assume the INAR(1) model where $\xi_{k}\equiv 0$ for every $k\geq 2$ and $\mathbb E[\xi_1]<1$.
\begin{equation}\label{eq:CLT2}
\frac{\sum_{k=1}^{n}X_{k}-n\mu}{\sqrt{n}}\rightarrow N\left(0, \frac{\mathbb{E}[\varepsilon]{\rm Var}[\xi_1]+{\rm Var}[\varepsilon](1-\mathbb{E}[\xi_1])}
{(1-\mathbb{E}[\xi_1])^3} \right)
\end{equation}
in distribution as $n\rightarrow \infty$.
\end{corollary}

\begin{remark}
It is easy to see that our Corollary~\ref{Coll2:CLT} recovers the central limit theorem[Theorem 3] in \cite{Pakes1971}. \end{remark}

\section{Proof of Main Results}\label{Proof of Main Results}

\subsection{Proof of large deviations}

\begin{proof}[Proof of Theorem~\ref{thm:LDP}]
For any $\theta\in\mathbb{R}$, we can compute that
\begin{align*}
&\mathbb{E}\left[e^{\theta\sum_{k=1}^{n}X_{k}}\right]
=\mathbb{E}\left[e^{\theta X_{n}+\theta\sum_{k=1}^{n-1}X_{k}}\right]
=\mathbb{E}\left[e^{\theta\varepsilon_{n}+\theta\sum_{k=1}^{n-1}\sum_{\ell=1}^{X_{n-k}}\xi_{\ell}^{(n,k)}+\theta\sum_{k=1}^{n-1}X_{n-k}}\right]
\\
=&\mathbb{E}[e^{\theta\varepsilon}]
\mathbb{E}\left[e^{\sum_{k=1}^{n-1}(\theta+\log\mathbb{E}[e^{\theta\xi_{k}}])X_{n-k}}\right]
=\mathbb{E}[e^{\theta\varepsilon}]
\mathbb{E}\left[e^{(\theta+\log\mathbb{E}[e^{\theta\xi_{1}}])X_{n-1}
+\sum_{k=2}^{n-1}(\theta+\log\mathbb{E}[e^{\theta\xi_{k}}])X_{n-k}}\right]\\
=&\mathbb{E}[e^{\theta\varepsilon}]
\mathbb{E}\left[e^{(\theta+\log\mathbb{E}[e^{\theta\xi_{1}}])\varepsilon_{n-1}
+(\theta+\log\mathbb{E}[e^{\theta\xi_{1}}])\sum_{k=1}^{n-2}\sum_{l=1}^{X_{n-1-k}}\xi_{l}^{(n-1,k)}
+\sum_{k=2}^{n-1}(\theta+\log\mathbb{E}[e^{\theta\xi_{k}}])X_{n-k}}\right]
\\
=&\mathbb{E}[e^{\theta\varepsilon}]
\mathbb{E}\left[e^{(\theta+\log\mathbb{E}[e^{\theta\xi_{1}}])\varepsilon}\right]
\mathbb{E}\left[e^{\sum_{k=1}^{n-2}\log\mathbb{E}[e^{(\theta+\log\mathbb{E}[e^{\theta\xi_{1}}])\xi_{k}}]X_{n-1-k}
+\sum_{k=2}^{n-1}(\theta+\log\mathbb{E}[e^{\theta\xi_{k}}])X_{n-k}}\right]
\\
=&\mathbb{E}[e^{\theta\varepsilon}]
\mathbb{E}\left[e^{(\theta+\log\mathbb{E}[e^{\theta\xi_{1}}])\varepsilon}\right]
\mathbb{E}\left[e^{\sum_{k=2}^{n-1}(\theta+\log\mathbb{E}[e^{\theta\xi_{k}}]+\log\mathbb{E}[e^{(\theta+
\log\mathbb{E}[e^{\theta\xi_{1}}])\xi_{k-1}}])X_{n-k}}\right].
\end{align*}
By iterations, we get for $\theta\in D_{\Lambda}:=(-\infty,\theta_c]$, where $\theta_c$ is defined in equation (\ref{maxF}),
\begin{equation}\label{moment generation}
\mathbb{E}\left[e^{\theta\sum_{k=1}^{n}X_{k}}\right]
=\prod_{k=1}^{n}\mathbb{E}\left[e^{f_{k}(\theta)\varepsilon}\right],
\end{equation}
where $f_{1}(\theta)=\theta$, and for any $k=2,3,\ldots,n$,
\begin{equation}\label{function f}
f_{k}(\theta)=\theta+\log\mathbb{E}[e^{f_{1}(\theta)\xi_{k-1}}]+\log\mathbb{E}[e^{f_{2}(\theta)\xi_{k-2}}]
+\cdots+\log\mathbb{E}[e^{f_{k-1}(\theta)\xi_{1}}].
\end{equation}
Hence, we have
\begin{equation}\label{Gamma}
\Gamma(\theta):=\lim_{n\rightarrow\infty}\frac{1}{n}\log\mathbb{E}\left[e^{\theta\sum_{k=1}^{n}X_{k}}\right]
=\log\mathbb{E}[e^{f_{\infty}(\theta)\varepsilon}],
\end{equation}
where $f_{\infty}(\theta)$ satisfies the equation:
\begin{equation}\label{finfty}
f_{\infty}(\theta)=\theta+\sum_{k=1}^{\infty}\log\mathbb{E}\left[e^{f_{\infty}(\theta)\xi_{k}}\right],
\end{equation}
provided that the equation $x=\theta+\sum_{k=1}^{\infty}\log\mathbb{E}[e^{x\xi_{k}}]$
has a solution and $\Gamma(\theta)=\infty$ otherwise.
Indeed, let us define the function
\begin{equation}\label{F}
F(x)=x-\sum_{k=1}^{\infty}\log\mathbb{E}[e^{x\xi_{k}}].
\end{equation}
It is easy to check that $F(x)$ is concave in $x$.
Then $f_{\infty}(\theta)$ is the smaller solution
to the equation $F(x)=\theta$
for any $\theta\leq\theta_{c}$, where
\begin{equation}\label{maxF}
\theta_{c}=\max_{x\in\mathbb{R}}F(x),
\end{equation}
and $f_{\infty}(\theta)$ is the limit of $f_k(\theta)$ in equation (\ref{function f}) as $k\rightarrow \infty$.

Next, let us show that the essential smoothness condition holds. That is $\theta_{c}>0$,
$\Gamma(\theta)$ is differentiable for any $\theta<\theta_{c}$ and $\Gamma'(\theta)\rightarrow\infty$
as $\theta\uparrow\theta_{c}$.
First, let us show that $\theta_{c}>0$.
Note that by Eq. \eqref{F}, $F(0)=0$ and $F'(0)=1-\sum_{k=1}^{\infty}\mathbb{E}[\xi_{k}]=1-\Vert\mathbb{E}[\boldsymbol{\xi}]\Vert_{1}>0$,
and thus $\theta_{c}=\max_{x\in\mathbb{R}}F(x)>0$.
Next, let us show that $\Gamma'(\theta)$ goes to $\infty$ as $\theta\uparrow\theta_{c}$.
From Eq. (\ref{finfty}) we can compute that $f'_{\infty}(\theta)=\frac{1}{1-\sum_{k=1}^{\infty}\frac{\mathbb{E}\left[e^{f_{\infty}(\theta)\xi_{k}}\xi_{k}\right]}
{\mathbb{E}\left[e^{f_{\infty}(\theta)\xi_{k}}\right]}}=\frac{1}{F'(f_{\infty}(\theta))}.$
From (\ref{F}) we know that for any $\theta\in\mathbb{R}$, $F(f_{\infty}(\theta))=f_{\infty}(\theta)-\sum_{k=1}^{\infty}\log\mathbb{E}[e^{f_{\infty}(\theta)\xi_{k}}]=\theta.$
Therefore, when $\theta=\theta_c$, by (\ref{maxF}), we have $F(f_{\infty}(\theta_c))=\theta_c=\max_{x\in\mathbb{R}}F(x)$. 
Since $F$ is concave, $F(f_{\infty}(\theta))$  tends to the maximum of $F$ as $\theta \uparrow \theta_c$, i.e.
$f_{\infty}(\theta)$ tends to the argmax of $F$ as $\theta \uparrow \theta_c$. We infer from  Eq. (\ref{F}) and (\ref{maxF}) that  $F'(f_{\infty}(\theta))$ goes to $0$ as $\theta\uparrow\theta_{c}$,
and hence $f'_{\infty}(\theta)$ goes to $\infty$ as $\theta\uparrow\theta_{c}$.
Finally, by differentiating Eq. (\ref{Gamma}) with respect to $\theta$, we conclude that
$$\Gamma'(\theta)=\frac{f'_{\infty}(\theta)\mathbb{E}\left[e^{f_{\infty}(\theta)\varepsilon}\varepsilon\right]}{\mathbb{E}
\left[e^{f_{\infty}(\theta)\varepsilon}\right]}\rightarrow \infty \quad {\rm as} \quad \theta\uparrow\theta_{c}.$$
Hence, the essential smoothness condition is satisfied. By G\"{a}rtner-Ellis theorem (see e.g. \cite{Dembo}),
$\mathbf{P}(\frac{1}{n}\sum_{k=1}^{n}X_{k}\in\cdot)$
satisfies a large deviation principle with the rate function
\begin{equation*}
I(x)=\sup_{\theta\leq\theta_{c}}\left\{\theta x-\log\mathbb{E}[e^{f_{\infty}(\theta)\varepsilon}]\right\}.
\end{equation*}
This completes the proof.
\end{proof}

\begin{remark} 
Although $F''\leq 0$, $F$ does not need to be strictly concave. When $\xi_k$ is a constant, there is only one solution of the equation $F(x)=\theta$ and the smaller solution is defined as the unique solution. When $\xi_k$ is a constant, this special case corresponds to the classical Cramer Theorem provide a reference for Cramer theorem. When $\varepsilon$ is identically zero, the dynamics is definite. We do not consider this trivial case.
\end{remark}

\begin{proof}[Proof of Collorary~\ref{Coll1:LDP}]
When $\varepsilon$ is Poisson with mean $\alpha_{0}$
and $\xi_{k}$ are Poisson with mean $\alpha_{k}$, $k=1,2,\ldots$,
we have
$\Gamma(\theta)=\log\mathbb{E}[e^{f_{\infty}(\theta)\varepsilon}]=\alpha_{0}(e^{f_{\infty}(\theta)}-1)$,
where $f_{\infty}(\theta)$ is the smaller solution
to the equation $F(x)=\theta$
for any $\theta\leq\theta_{c}:=\max_{x\in\mathbb{R}}F(x)=\Vert\alpha\Vert_{1}-\log\Vert\alpha\Vert_{1}-1$,
where
\begin{equation*}
F(x)=x-\sum_{k=1}^{\infty}\log\mathbb{E}[e^{x\xi_{k}}]
=x-\Vert\alpha\Vert_{1}(e^{x}-1),
\end{equation*}
and by applying Theorem~\ref{thm:LDP} we complete the proof.
\end{proof}

\begin{proof}[Proof of Collorary~\ref{Coll2:LDP}]
Consider INAR(1) model
where we assume $\xi_{k}\equiv 0$ for every $k\geq 2$.
In this case,
\begin{equation*}
I(x)=\sup_{\theta\leq\theta_{c}}\left\{\theta x-\log\mathbb{E}[e^{f_{\infty}(\theta)\varepsilon}]\right\},
\end{equation*}
where $f_{\infty}(\theta)$ satisfies the equation $f_{\infty}(\theta)=\theta+\log\mathbb{E}\left[e^{f_{\infty}(\theta)\xi_{1}}\right]$ and $\theta_{c}=\max_{x}\left(x-\log\mathbb{E}[e^{x\xi_{1}}]\right)$. By letting $\psi=f_{\infty}(\theta)$, we have
\begin{equation*}
I(x)=\sup_{\psi\in\mathbb{R}}\left\{x\left(\psi-\log\mathbb{E}[e^{\psi\xi_{1}}]\right)-\log\mathbb{E}[e^{\psi\varepsilon}]\right\}.
\end{equation*}
\end{proof}

\subsection{Proof of moderate deviations}\label{Proof of MD}

\begin{proof}[Proof of Theorem~\ref{Thm:MDP}]
First, for any $\theta\in\mathbb{R}$, let us prove that
$$\lim_{n\to\infty}\dfrac{n}{c^2(n)}\log{\mathbb{E}\left[e^{\frac{c(n)}{n}\theta \left(\sum_{k=1}^{n}X_{k}-n\mu\right)}\right]}
=\frac{\theta^2(\mathbb{E}[\varepsilon]\Vert{\rm Var}[\boldsymbol{\xi}]\Vert_{1}+{\rm Var}[\varepsilon](1-\Vert\mathbb{E}[\boldsymbol{\xi}]\Vert_{1}))}
{2(1-\Vert\mathbb{E}[\boldsymbol{\xi}]\Vert_{1})^3},$$
where $\mu$ is defined by equation~\eqref{eq:lln}. From equation (\ref{moment generation}) in Theorem \ref{thm:LDP} , we get
\begin{equation}\label{moment generation1}
\mathbb{E}\left[e^{\theta\sum_{k=1}^{n}X_{k}}\right]
=\prod_{k=1}^{n}\mathbb{E}\left[e^{f_{k}(\theta)\varepsilon}\right],
\end{equation}
where $f_{1}(\theta)=\theta$, and for any $k=2,3,\ldots,n$,
\begin{equation*}
f_{k}(\theta)=\theta+\log\mathbb{E}[e^{f_{1}(\theta)\xi_{k-1}}]+\log\mathbb{E}[e^{f_{2}(\theta)\xi_{k-2}}]
+\cdots+\log\mathbb{E}[e^{f_{k-1}(\theta)\xi_{1}}].
\end{equation*}
Let us define $G_n(k)=f_k(\theta_n)$ where $\theta_n=\frac{c(n)}{n}\theta$ so that
\begin{equation}
\begin{aligned}\label{G_n}
G_n(k)&=\theta_n+\log\mathbb{E}[e^{G_n(1)\xi_{k-1}}]+\log\mathbb{E}[e^{G_n(2)\xi_{k-2}}]
+\cdots+\log\mathbb{E}[e^{G_n(k-1)\xi_{1}}]\\
      &=\theta_n+\sum_{i=1}^{k-1}\log\mathbb{E}[e^{G_n(k-i)\xi_{i}}].
\end{aligned}
\end{equation}
Then from (\ref{moment generation1}) we have
\begin{align}\label{eq:char Nt}
    \mathbb{E}\left[e^{\frac{c(n)}{n}\theta\sum_{k=1}^{n}X_{k}}\right]
    &=\prod_{i=1}^n\mathbb{E}\left[e^{\varepsilon{G_n(k)}}\right].
\end{align}
We write $G_n(k)$ instead of $G(k)$ to emphasize its dependence on $n$ because of the term $c(n)/n$.
As the proof of Theorem~\ref{thm:LDP} shows and $\frac{c(n)}{n}$ is sufficiently small so that we have $\frac{c(n)}{n}\theta\leq \theta_c $
where $\theta_c=\max_{x\in\mathbb{R}}F(x)$ and $F(x)=x-\sum_{k=1}^{\infty}\log\mathbb{E}[e^{x\xi_{k}}]$ and as $k\to \infty$, we get that
$G_n(\infty)=f_{\infty}(\theta_n)$ is the smaller solution to the equation 
\begin{equation}\label{thetan}
x_n=\theta_n+\sum^{\infty}_{i=1}\log\mathbb{E}[e^{x_n\xi_{i}}].	
\end{equation}
Recall equation (\ref{finfty}) in the proof of Thm \ref{thm:LDP}, takes $\theta=\theta_n$ we can get equation (\ref{thetan}).
By the assumption $\Vert\mathbb{E}[\boldsymbol{\xi}]\Vert_{1}<1$, it is easy to see that $x_n=O(c(n)/n)$ as $n\rightarrow\infty$. Because $x_n=O(c(n)/n)$ as $n\rightarrow\infty$,
we have $G_n(k)=O(c(n)/n)$ uniformly in $k$ as $n\rightarrow\infty$. By Taylor's expansion, we know that
\begin{align*}
&\log\mathbb{E}\left[e^{G_n(1)\xi_{k-1}}\right]=\log\mathbb{E}\left[1+{G_n(1)\xi_{k-1}}+\frac{G_n^2(1)\xi_{k-1}^2}{2!}\right]+O\left(\left(\dfrac{c(n)}{n}\right)^3\right)\\
=&{G_n(1)}\mathbb{E}[\xi_{k-1}]+\frac{1}{2}{G_n^2(1)}\mathbb{E}[\xi_{k-1}^2]-
\frac{1}{2}{G_n^2(1)}(\mathbb{E}[\xi_{k-1}])^2+O\left(\left(\dfrac{c(n)}{n}\right)^3\right).
\end{align*}
The last equation depends on the fact that $G_n(k)=O(c(n)/n)$ uniformly in $k$ as $n\rightarrow\infty$. 
Then from (\ref{G_n}) we have:
\begin{equation}
\begin{aligned}\label{G_n1}
G_n(k)&=\theta_n+(G_n(1)\mathbb{E}[\xi_{k-1}]+\cdots+G_n(k-1)\mathbb{E}[\xi_{1}])+\frac{1}{2}(G_n^2(1)\mathbb{E}[\xi_{k-1}^2]+\cdots+G_n^2(k-1)\mathbb{E}[\xi_{1}^2])\\
&-\frac{1}{2}(G_n^2(1)(\mathbb{E}[\xi_{k-1}])^2+\cdots+G_n^2(k-1)(\mathbb{E}[\xi_{1}])^2)+O\left(\left(\dfrac{c(n)}{n}\right)^3\right)\\
&=\theta_n+\sum\limits_{i=1}^{k-1}\mathbb{E}[\xi_{i}]G_n(k-i)+\frac{1}{2}\sum\limits_{i=1}^{k-1}{\rm Var}
[\xi_{i}]G_n^2(k-i)+O\left(\left(\dfrac{c(n)}{n}\right)^3\right).
\end{aligned}
\end{equation}
Then from (\ref{eq:char Nt}) for any $\theta\in\mathbb{R}$,
\begin{equation}\label{eq:char Nt1}
\begin{aligned}
&\frac{n}{c^2(n)}\log{\mathbb{E}\left[e^{\frac{c(n)}{n}\theta \left(\sum\limits_{k=1}^{n}X_{k}-n\mu\right)}\right]}=\frac{n}{c^2(n)}\sum\limits_{k=1}^{n}\log{\mathbb{E}\left[e^{\varepsilon G_n(k)}\right]}-\frac{n\mu\theta}{c(n)}\\
=&\frac{n}{c^2(n)}\sum\limits_{k=1}^{n}\left(G_n(k)\mathbb{E}[\varepsilon]+\frac{1}{2}G_n^2(k){\rm Var}
[\varepsilon]\right)-\frac{n\mu\theta}{c(n)}+O\left(\left(\frac{c(n)}{n}\right)^2\right).
\end{aligned}
\end{equation}

Next, we want to do the second order asymptotic expansion of the above moment generating function, and prove the convergence of the first-order and the second-order terms, 
whereas the higher order terms can be ignored. Let us write
\begin{equation}\label{G_n2}
G_n(k)=\left(\frac{c(n)}{n}\theta\right)\overline{G}_1(k)+
\left(\frac{c(n)}{n}\theta\right)^2\overline{G}_2(k)+\varepsilon_{n}(k),
\end{equation}
where $\overline{G}_1(k)$ and $\overline{G}_2(k)$ satisfy the following equations:
\begin{equation}\label{G_1(k)1}
\overline{G}_1(k)=1+\sum\limits_{i=1}^{k-1}\mathbb{E}[\xi_{i}]\overline{G}_1(k-i)\qquad \rm with \qquad  \overline{G}_1(1)=1,
\end{equation}
\begin{equation} \label{G_2(k)1}
\overline{G}_2(k)=\sum\limits_{i=1}^{k-1}\mathbb{E}[\xi_{i}]\overline{G}_2(k-i)
+\frac{1}{2}\sum\limits_{i=1}^{k-1}{\rm Var}[\xi_{i}]\overline{G}_1^2(k-i)\qquad \rm with  \qquad  \overline{G}_2(1)=0,
\end{equation}
and note that $\varepsilon_{n}(k)$ is the remainder term of the second order asymptotic expansion. Here the equations (\ref{G_1(k)1}) and (\ref{G_2(k)1}) are obtained by substituting (\ref{G_n2}) into (\ref{G_n1}), we get that
\begin{align*}
G_n(k)=&\frac{c(n)}{n}\theta+\sum\limits_{i=1}^{k-1}{\mathbb{E}[\xi_{i}]\left(\frac{c(n)}{n}\theta{\overline{G}_1(k-i)}
+\left(\frac{c(n)}{n}\theta\right)^2{\overline{G}_2(k-i)}+\varepsilon_{n}(k-i)\right)}
\\
&+\frac{1}{2}\sum\limits_{i=1}^{k-1}{{\rm Var}[\xi_{i}]\left(\frac{c(n)}{n}\theta{\overline{G}_1(k-i)}
+\left(\frac{c(n)}{n}\theta\right)^2{\overline{G}_2(k-i)}+\varepsilon_{n}(k-i)\right)^2}\\
=&\frac{c(n)}{n}\theta\left(1+\sum\limits_{i=1}^{k-1}{\mathbb{E}[\xi_{i}]{\overline{G}_1(k-i)}}\right)
\\
&+\left(\frac{c(n)}{n}\theta\right)^2\left(\sum\limits_{i=1}^{k-1}{\mathbb{E}[\xi_{i}]{\overline{G}_2(k-i)}}
+\frac{1}{2}\sum\limits_{i=1}^{k-1}{{\rm Var}[\xi_{i}]{\overline{G}_1^2(k-i)}}\right)+O\left(\sum\limits_{i=1}^{k-1}\varepsilon_{n}(k-i)\right),
\end{align*}
then we get the equations of $\overline{G}_1(k)$ and $\overline{G}_2(k)$ and also from the fact that $G_n(k)=O(c(n)/n)$ uniformly in $k$ as $n\rightarrow\infty$ we know that $\varepsilon_{n}(k)=O\left(\left(\dfrac{c(n)}{n}\right)^3\right)$ uniformly in $k$ as $n\rightarrow\infty$. From (\ref{eq:char Nt1}) and (\ref{G_n2}), we obtain
\begin{equation}\label{MGF1}
\begin{aligned}
&\frac{n}{c^2(n)}\log\mathbb{E}\left[e^{\frac{c(n)}{n}\theta\left(\sum\limits_{k=1}^{n}{X_{k}-n\mu}\right)}\right]\\
=&\frac{n}{c^2(n)}\sum\limits_{k=1}^{n}\bigg[\mathbb{E}[\varepsilon]\left(\frac{c(n)}{n}\theta{\overline{G}_1(k)}
+\left(\frac{c(n)}{n}\theta\right)^2{\overline{G}_2(k)}\right)\\
&+\frac{{\rm Var}[\varepsilon]}{2}\left(\frac{c(n)}{n}
\theta{\overline{G}_1(k)}+\left(\frac{c(n)}{n}\theta\right)^2{\overline{G}_2(k)}\right)^2\bigg]
-\frac{n\mu\theta}{c(n)}+O\left(\frac{c(n)}{n^2}\right)
\\=&\frac{\mathbb{E}[\varepsilon]\theta}{c(n)}\sum\limits_{k=1}^{n}{\overline{G}_1(k)}-\frac{n\mu\theta}{c(n)}
+\frac{{\theta}^2}{n}\left(\mathbb{E}[\varepsilon]\sum\limits_{k=1}^{n}{\overline{G}_2(k)}+
\frac{{\rm Var}[\varepsilon]}{2}\sum\limits_{k=1}^{n}\overline{G}_1^2(k)\right)+O\left(\frac{c(n)}{n^2}\right).
\end{aligned}
\end{equation}
Next, let us compute $\sum\limits_{k=1}^{n}{\overline{G}_1(k)}$. By (\ref{G_1(k)1}),
\begin{align*}
\sum\limits_{k=1}^{n}{\overline{G}_1(k)}&=\sum\limits_{k=2}^{n}{\overline{G}_1(k)}+1=1+\sum\limits_{k=2}^{n}\left({1+\sum\limits_{i=1}^{k-1}{\mathbb{E}[\xi_{i}]\overline{G}_1(k-i)}}\right)\\
&=n+\sum\limits_{k=2}^{n}\sum\limits_{i=1}^{k-1}{{\mathbb{E}[\xi_{i}]\overline{G}_1(k-i)}}=n+\sum\limits_{i=1}^{n-1}\sum\limits_{k=i+1}^{n}{{\mathbb{E}[\xi_{i}]\overline{G}_1(k-i)}},
\end{align*}
let $j=k-i$, then $1\leqslant{j}\leqslant{n-i}$, hence
\begin{align*}
\sum\limits_{k=1}^{n}{\overline{G}_1(k)}&=n+\sum\limits_{i=1}^{n-1}\sum\limits_{j=1}^{n-i}
{\mathbb{E}[\xi_{i}]\overline{G}_1(j)}=n+\sum\limits_{i=1}^{n-1}{\mathbb{E}[\xi_{i}]\left(\sum\limits_{j=1}^{n}\overline{G}_1(j)-
\sum\limits_{j=n-i+1}^{n}\overline{G}_1(j)\right)},
\end{align*}
\begin{equation*}
\sum\limits_{k=1}^{n}\overline{G}_1(k)-\sum\limits_{i=1}^{n-1}\mathbb{E}[\xi_{i}]\sum\limits_{j=1}^{n}
\overline{G}_1(j)=n-\sum\limits_{i=1}^{n-1}\mathbb{E}[\xi_{i}]\sum\limits_{j=n-i+1}^{n}\overline{G}_1(j).
\end{equation*}
Hence
\begin{equation*}
\sum\limits_{k=1}^{n}\overline{G}_1(k)=\frac{n-\sum\limits_{i=1}^{n-1}\mathbb{E}[\xi_{i}]
\sum\limits_{j=n-i+1}^{n}\overline{G}_1(j)}{1-\sum\limits_{i=1}^{n-1}\mathbb{E}[\xi_{i}]}.
\end{equation*}
Then
\begin{equation}\label{G_3(k)1}
\begin{aligned}
&\frac{\mathbb{E}[\varepsilon]\theta}{c(n)}\sum\limits_{k=1}^{n}{\overline{G}_1(k)}-\frac{n\mu\theta}{c(n)}=\frac{\mathbb{E}[\varepsilon]\theta}{c(n)}\left(\frac{n-\sum\limits_{i=1}^{n-1}{\mathbb{E}[\xi_{i}]}
\sum\limits_{j=n-i+1}^{n}{\overline{G}_1(j)}}{1-\sum\limits_{i=1}^{n-1}{\mathbb{E}[\xi_{i}]}}-\frac{n}
{1-\Vert\mathbb{E}[\boldsymbol{\xi}]\Vert_{1}}\right)\\
=&\frac{\mathbb{E}[\varepsilon]\theta}{c(n)}\left(\frac{n}{1-\sum\limits_{i=1}^{n-1}{\mathbb{E}[\xi_{i}]}}-
\frac{n}{1-\Vert\mathbb{E}[\boldsymbol{\xi}]\Vert_{1}}-\frac{\sum\limits_{i=1}^{n-1}{\mathbb{E}[\xi_{i}]}
\sum\limits_{j=n-i+1}^{n}{\overline{G}_1(j)}}{1-\sum\limits_{i=1}^{n}{\mathbb{E}[\xi_{i}]}}\right).
\end{aligned}
\end{equation}
For the first term in equation (\ref{G_3(k)1}), we can compute
\begin{align*}
&\left\vert\frac{\mathbb{E}[\varepsilon]\theta}{c(n)}\left(\frac{n}{1-\sum\limits_{i=1}^{n-1}{\mathbb{E}[\xi_{i}]}}-
\frac{n}{1-\Vert\mathbb{E}[\boldsymbol{\xi}]\Vert_{1}}\right)\right\vert
=\left\vert\frac{\mathbb{E}[\varepsilon]n\theta}{c(n)}\left(\frac{\sum\limits_{i=1}^{n-1}{\mathbb{E}[\xi_{i}]}
-\Vert\mathbb{E}[\boldsymbol{\xi}]\Vert_{1}}{\left(1-\sum\limits_{i=1}^{n-1}{\mathbb{E}[\xi_{i}]}\right)(1-\Vert\mathbb{E}[\boldsymbol{\xi}]\Vert_{1})}
\right)\right\vert\\
\leqslant &\left\vert\frac{\mathbb{E}[\varepsilon]n\theta}{c(n)}\frac{\sum\limits_{i=n}^{\infty}{\mathbb{E}[\xi_{i}]}}
{(1-\Vert\mathbb{E}[\boldsymbol{\xi}]\Vert_{1})^2}\right\vert=\frac{\mathbb{E}[\varepsilon]{n}}{c(n)}\left\vert\frac{\theta\sum\limits_{i=n}^{\infty}
{\mathbb{E}[\xi_{i}]}}{(1-\Vert\mathbb{E}[\boldsymbol{\xi}]\Vert_{1})^2}\right\vert.
\end{align*}
Under the assumption $\sup\limits_{n>0}n^\frac{3}{2}\mathbb{E}[\xi_{n}]=C_1<\infty$, we have
\begin{equation}\label{sum control}
\sum\limits_{i=n}^{\infty}{\mathbb{E}[\xi_{i}]}\leqslant\sum\limits_{i=n}^{\infty}{C_1{i}^{-\frac{3}{2}}}
<2C_1(n-1)^{-\frac{1}{2}}.
\end{equation}
Therefore, we conclude that
\begin{equation}\label{sum convergence}
\frac{\mathbb{E}[\varepsilon]\theta{n}}{c(n)}\left\vert\frac{\sum\limits_{i=n}^{\infty}
{\mathbb{E}[\xi_{i}]}}{(1-\Vert\mathbb{E}[\boldsymbol{\xi}]\Vert_{1})^2}\right\vert \leqslant\frac{\vert\mathbb{E}[\varepsilon]
\vert\vert\theta\vert{2C_1n}}{c(n)\sqrt{n-1}\vert(1-\Vert\mathbb{E}[\boldsymbol{\xi}]\Vert_{1})^2\vert}\rightarrow0,
\end{equation}
as
$n\rightarrow\infty$.
By Lemma~\ref{Lem1:MDP}, $\overline{G}_1(t)$ is uniformly bounded, then we write $\overline{G}_1(\infty)$ for a upper bound of $\overline{G}_1(t)$ over $t$. Then for the second term in equation (\ref{G_3(k)1}), we have
\begin{align*}
&\limsup_{n\rightarrow\infty}\left\vert\frac{\mathbb{E}[\varepsilon]\theta}{c(n)}\frac{\sum\limits_{i=1}^{n-1}
{\mathbb{E}[\xi_{i}]}\sum\limits_{j=n-i+1}^{n}{\overline{G}_1(j)}}{1-\sum\limits_{i=1}^{n}{\mathbb{E}[\xi_{i}]}}\right\vert
\leqslant {\overline{G}_1(\infty)}\limsup_{n\rightarrow\infty}\frac{\mathbb{E}[\varepsilon]\theta}{c(n)}
\left\vert\frac{\sum\limits_{i=1}^{n-1}{(i-1)\mathbb{E}[\xi_{i}]}}{1-\Vert\mathbb{E}[\boldsymbol{\xi}]\Vert_{1}}\right\vert
\\
\leqslant & {\overline{G}_1(\infty)}\limsup_{n\rightarrow\infty}\frac{\mathbb{E}[\varepsilon]\theta}{c(n)}
\left\vert\frac{\sum\limits_{i=1}^{n-1}{\frac{C_1}{\sqrt{i}}}}{1-\Vert\mathbb{E}[\boldsymbol{\xi}]\Vert_{1}}\right\vert
\leqslant {\overline{G}_1(\infty)}\limsup_{n\rightarrow\infty}\frac{C_1\mathbb{E}[\varepsilon]\theta}
{1-\Vert\mathbb{E}[\boldsymbol{\xi}]\Vert_{1}}\frac{2\sqrt{n}}{c(n)}.
\end{align*}
Hence, $\lim\limits_{n\rightarrow\infty}\left[\frac{\mathbb{E}[\varepsilon]\theta}{c(n)}\sum\limits_{k=1}^{n}
{\overline{G}_1(k)}-\frac{n\mu\theta}{c(n)}\right]=0.$
Furthermore, we can compute that
\begin{equation}\label{G_1bar}
\lim\limits_{n\rightarrow\infty}\frac{1}{n}\sum\limits_{k=1}^{n}{\overline{G}_1(k)}=\frac{1}{1-\Vert\mathbb{E}[\boldsymbol{\xi}]\Vert_{1}}.
\end{equation}

Next, we compute the limits of $\frac{1}{n}\sum\limits_{k=1}^{n}{\overline{G}_1^2(k)}$ and $\frac{1}{n}\sum\limits_{k=1}^{n}{\overline{G}_2(k)}$ as $n\rightarrow\infty$. According to Lemma~\ref{Lem2:MDP}, $\overline{G}_2(n)$ is uniformly bounded in $n$.
We know from (\ref{G_1(k)1}) that
\begin{align*}
\sum\limits_{k=1}^{n}{\overline{G}^2_1(k)}&=1+\sum\limits_{k=2}^{n}{\overline{G}^2_1(k)}=1+\sum\limits_{k=2}^{n}\left({1+\sum\limits_{i=1}^{k-1}{\mathbb{E}[\xi_{i}]\overline{G}_1(k-i)}}\right)^2
\\
&=n+2\sum\limits_{k=2}^{n}\sum\limits_{i=1}^{k-1}{{\mathbb{E}[\xi_{i}]\overline{G}_1(k-i)}}
+\sum\limits_{k=2}^{n}\left(\sum\limits_{i=1}^{k-1}\mathbb{E}[\xi_{i}]\overline{G}_1(k-i)\right)^2,
\end{align*}
let $j=k-i$, then $1\leqslant{j}\leqslant{n-i}$, hence as $n\rightarrow \infty$ and by equation (\ref{G_1bar})
\begin{align*}
\frac{\sum\limits_{k=2}^{n}\sum\limits_{i=1}^{k-1}{{\mathbb{E}[\xi_{i}]\overline{G}_1(k-i)}}}{n}
=\frac{\sum\limits_{i=1}^{n-1}\sum\limits_{j=1}^{n-i}
{{\mathbb{E}[\xi_{i}]\overline{G}_1(j)}}}{n}
=\sum\limits_{i=1}^{n-1}
{\mathbb{E}[\xi_{i}]\frac{\sum\limits_{j=1}^{n-i}\overline{G}_1(j)}{n}\rightarrow\frac{\Vert\mathbb{E}[\boldsymbol{\xi}]\Vert_{1}}{1-\Vert\mathbb{E}[\boldsymbol{\xi}]\Vert_{1}}},
\end{align*}
Then we have that as $n\rightarrow \infty$,
$$\frac{\sum\limits_{k=2}^{n}\left(\sum\limits_{i=1}^{k-1}\mathbb{E}[\xi_{i}]\overline{G}_1(k-i)\right)^2}{n}\rightarrow \frac{\Vert\mathbb{E}[\boldsymbol{\xi}]\Vert_{1}^2}{(1-\Vert\mathbb{E}[\boldsymbol{\xi}]\Vert_{1})^2}.$$
That means
\begin{equation}\label{sum G_1^2}
\lim_{n\rightarrow\infty}\frac{1}{n}\sum\limits_{k=1}^{n}{\overline{G}_1^2(k)}=1+\frac{2\Vert\mathbb{E}[\boldsymbol{\xi}]
\Vert_{1}}{1-\Vert\mathbb{E}[\boldsymbol{\xi}]\Vert_{1}}+\frac{\Vert\mathbb{E}[\boldsymbol{\xi}]\Vert_{1}^2}{(1-\Vert\mathbb{E}[\boldsymbol{\xi}]
\Vert_{1})^2}=\frac{1}{(1-\Vert\mathbb{E}[\boldsymbol{\xi}]\Vert_{1})^2}.
\end{equation}
For $\frac{1}{n}\sum\limits_{k=1}^{n}{\overline{G}_2(k)}$ from (\ref{G_2(k)1}) we get that
$$\frac{\sum\limits_{k=1}^{n}\overline{G}_2(k)}{n}=\frac{\sum\limits_{k=1}^{n}\sum\limits_{i=1}^{k-1}\mathbb{E}[\xi_{i}]\overline{G}_2(k-i)}{n}
+\frac{\frac{1}{2}\sum\limits_{k=1}^{n}\sum\limits_{i=1}^{k-1}{\rm Var}[\xi_{i}]\overline{G}_1^2(k-i)}{n},$$
That means
\begin{equation}\label{sum G_2}
\lim_{n\rightarrow\infty}\frac{1}{n}\sum\limits_{k=1}^{n}{\overline{G}_2(k)}=\frac{\Vert{\rm Var}[\boldsymbol{\xi}]\Vert_{1}}{2(1-\Vert\mathbb{E}[\boldsymbol{\xi}]\Vert_{1})^3}.
\end{equation}

Thus, we have
\begin{equation*}
\lim_{n\rightarrow\infty}\frac{\theta^2}{n}\left(\mathbb{E}[\varepsilon]\sum\limits_{k=1}^{n}{\overline{G}_2(k)}
+\frac{{\rm Var}[\varepsilon]}{2}\sum\limits_{k=1}^{n}\overline{G}_1^2(k)\right)=\frac{\theta^2
(\mathbb{E}[\varepsilon]\Vert{\rm Var}[\boldsymbol{\xi}]\Vert_{1}+{\rm Var}[\varepsilon](1-\Vert\mathbb{E}[\boldsymbol{\xi}]\Vert_{1}))}
{2(1-\Vert\mathbb{E}[\boldsymbol{\xi}]\Vert_{1})^3}.
\end{equation*}
Therefore, we have proved that
\begin{equation*}
\lim_{n\rightarrow\infty}\frac{n}{c^2(n)}\log\mathbb{E}\left[e^{\frac{c(n)}{n}\theta\left(\sum\limits_{k=1}^{n}
{X_{k}-n\mu}\right)}\right]=\frac{\theta^2(\mathbb{E}[\varepsilon]\Vert{\rm Var}[\boldsymbol{\xi}]\Vert_{1}+{\rm Var}[\varepsilon](1-\Vert\mathbb{E}[\boldsymbol{\xi}]\Vert_{1}))}
{2(1-\Vert\mathbb{E}[\boldsymbol{\xi}]\Vert_{1})^3}.
\end{equation*}
By applying the G\"{a}rtner-Ellis theorem (see e.g. \cite{Dembo}), we conclude that, for any Borel set $A$,
\begin{align*}
-\inf_{x\in A^{o}}J(x)
&\leq\liminf_{n\rightarrow\infty}\frac{n}{c^2(n)}\log \mathbf{P}\left(\dfrac{\sum_{k=1}^{n}X_{k}-n\mu}{c(n)}\in A\right)\\
&\leq\limsup_{n\rightarrow\infty}\frac{n}{c^2(n)}\log \mathbf{P}\left(\dfrac{\sum_{k=1}^{n}X_{k}-n\mu}{c(n)}\in A\right)\leq-\inf_{x\in\overline{A}}J(x),
\end{align*}
where
\begin{align*}
    J(x)&=\sup_{\theta\in\mathbb{R}}\left\{\theta x-\frac{\theta^2(\mathbb{E}[\varepsilon]\Vert{\rm Var}[\boldsymbol{\xi}]\Vert_{1}+{\rm Var}[\varepsilon](1-\Vert\mathbb{E}[\boldsymbol{\xi}]\Vert_{1})}
{2(1-\Vert\mathbb{E}[\boldsymbol{\xi}]\Vert)^3}\right\}\\
    &=\frac{x^2(1-\Vert\mathbb{E}[\boldsymbol{\xi}]\Vert_{1})^3}
{2(\mathbb{E}[\varepsilon]\Vert{\rm Var}[\boldsymbol{\xi}]\Vert_{1}+{\rm Var}[\varepsilon](1-\Vert\mathbb{E}[\boldsymbol{\xi}]\Vert_{1}))}.
\end{align*}
This completes the proof.
\end{proof}
\begin{remark} 
The equation (\ref{thetan}) can read as $F(x_n)=\theta_n$ as equation (\ref{F}). Since $F$ is concave, the equation may have two solutions. Define $G_n(\infty)$ as the smaller solution. By Tayor's formula for equation (\ref{thetan}), we get that
\begin{align*}
x_n&=\theta_n+\sum^{\infty}_{i=1}\log\mathbb{E}[e^{x_n\xi_{i}}]=\theta_n+\sum^{\infty}_{i=1} \log[1+\mathbb{E}[\xi_i]x_n+\mathbb{E}[\xi_i]O(x_n^2)]\\
   &=\theta_n+\sum^{\infty}_{i=1}\mathbb{E}[\xi_i]x_n+O(x_n^2).
   \end{align*}
Since $\Vert\mathbb{E}[\boldsymbol{\xi}]\Vert_{1}<1$, then we arrive at $x_n=O(\theta_n)=O(c(n)/n)$.
\end{remark}

\begin{proof}[Proof of Corollary~\ref{Coll1:MDP}]
When $\varepsilon$ is Poisson with mean $\alpha_{0}$, $\xi_{k}$ are Poisson with mean $\alpha_{k}$, $k=1,2,\ldots$,
we can compute
$\Vert\mathbb{E}[\boldsymbol{\xi}]\Vert_{1}$ and $\Vert{\rm Var}[\boldsymbol{\xi}]\Vert_{1}$, where $\Vert\mathbb{E}[\boldsymbol{\xi}]\Vert_{1}=\sum_{i=1}^{\infty}\mathbb{E}[{\xi}_{i}]=
    \sum_{i=1}^{\infty}\alpha_{i}=\Vert \alpha\Vert_{1},
    \Vert{\rm Var}[\boldsymbol{\xi}]\Vert_{1}=\sum_{i=1}^{\infty}\rm Var[{\xi}_{i}]=\sum_{i=1}^{\infty}\alpha_{i}=\Vert \alpha\Vert_{1}$. By applying Theorem \ref{Thm:MDP} we have
\begin{align*}
J(x)=\frac{x^2(1-\Vert \alpha\Vert_{1})^3}
{2\alpha_{0}\left(\Vert \alpha\Vert_{1}+(1-\Vert \alpha\Vert_{1})\right)}
=\dfrac{x^2\left(1-\Vert \alpha\Vert_{1}\right)^3}{2\alpha_0}.
\end{align*}
This completes the proof.
\end{proof}

\begin{proof}[Proof of Corollary~\ref{Coll2:MDP}]
Consider INAR(1) model where we assume $\xi_{k}\equiv 0$ for every $k\geq 2$. Then we can compute $\Vert\mathbb{E}[\boldsymbol{\xi}]\Vert_{1}=\sum_{i=1}^{\infty}\mathbb{E}[{\xi}_{i}]=
    \mathbb{E}[{\xi}_{1}], 
    \Vert\rm Var[\boldsymbol{\xi}]\Vert_{1}=\sum_{i=1}^{\infty}\rm Var[{\xi}_{i}]=\rm Var[{\xi}_{1}]$. By applying Theorem~\ref{Thm:MDP} we have
\begin{align*}
J(x)=\frac{x^2(1-\mathbb{E}[{\xi}_{1}])^3}
{2(\mathbb{E}[\varepsilon]\rm Var[{\xi}_{1}]+\rm Var[\varepsilon](1-\mathbb{E}[{\xi}_{1}]))}.
\end{align*}
This completes the proof.
\end{proof}

\subsection{Proof of law of large numbers}\label{Proof of LLN}

\begin{proof}[Proof of Theorem~\ref{thm:LLN}]
Let $\mathcal{F}_n$ be the $\sigma$-algebra generated by $X_1, X_2,\ldots,X_n$. Then, it is easy to see that
\begin{equation}\label{martingale}
M_n:=\sum_{i=1}^n(X_i-\mathbb{E}[X_i|\mathcal{F}_{i-1}])
\end{equation}
is a martingale with respect to  $(\mathcal{F}_n)^{\infty}_{n=1}$ and the $\sigma$-field $\mathcal{F}_0$ is the trivial $\sigma$-field containing only the empty set $\emptyset$ and the sample space, on which the sequence $(X_i)_{i=1}^{\infty}$ is defined. On the other hand, from the definition of $(X_i)^{\infty}_{i=1}$, it is clear that

\begin{equation}\label{conditional-expection}
\mathbb{E}[X_i|\mathcal{F}_{i-1}]=\mathbb{E}\left[\left(\varepsilon_i+\sum_{k=1}^{i-1}\sum_{l=1}^{X_{i-k}}\xi_{l}^{(i,k)}\right)|\mathcal{F}_{i-1}\right]
                                 =\mathbb{E}[\varepsilon_i]+\sum_{k=1}^{i-1}\sum_{l=1}^{X_{i-k}}\mathbb{E}[\xi_{l}^{(i,k)}]
                                 =\mathbb{E}[\varepsilon]+\sum_{k=1}^{i-1}\sum_{l=1}^{X_{i-k}}\mathbb{E}[\xi_k]. 
\end{equation}
Therefore,
\begin{align}\label{martingale decomposition}
M_n&=\sum_{i=1}^n(X_i-\mathbb{E}[X_i|\mathcal{F}_{i-1}])\nonumber=\sum_{i=1}^nX_i-\sum_{i=1}^n\mathbb{E}[X_i|\mathcal{F}_{i-1}]\nonumber\\
   &=\sum_{i=1}^nX_i-n\mathbb{E}[\varepsilon]-\sum_{i=1}^n\sum_{k=1}^{i-1}\mathbb{E}[\xi_k]X_{i-k}\nonumber=\sum_{i=1}^nX_i-n\mathbb{E}[\varepsilon]-\sum_{i=1}^{n-1}\sum_{k=1}^{n-i}\mathbb{E}[\xi_k]X_i\nonumber\\
   &=\left(1-\Vert\mathbb{E}[\boldsymbol{\xi}]\Vert_{1}\right)\sum_{i=1}^nX_i-n\mathbb{E}[\varepsilon]+\epsilon_n,
\end{align}
where
\begin{equation*}
\epsilon_n:=\sum_{i=1}^{n-1}\sum_{k=n-i+1}^{\infty}\mathbb{E}[\xi_k]X_i+\Vert\mathbb{E}[\boldsymbol{\xi}]\Vert_{1}X_n \ge 0.
\end{equation*}

First we bound the second term of $\epsilon_n$. We want to show that $\frac{X_n}{n}\rightarrow 0$ in probability as $n\rightarrow \infty$. So we want to prove that $\mathbb{E}[X_i]$ is uniformly bounded, then by Markov inequality we can get the desired result. By taking the expectation in equation \eqref{main:dynamicsempty}, we get that 
\begin{align*}
\mathbb{E}[X_i]=\mathbb{E}[\varepsilon]+\sum_{k=1}^{i-1}\mathbb{E}[\xi_k]\mathbb{E}[X_{i-k}]&\leq \mathbb{E}[\varepsilon]+\max_{1\leq k\leq i-1}\mathbb{E}[X_{i-k}]\sum_{k=1}^{i-1}\mathbb{E}[\xi_k]\\
               &\leq \mathbb{E}[\varepsilon]+\Vert\mathbb{E}[\boldsymbol{\xi}]\Vert_{1}\max_{1\leq k\leq i}\mathbb{E}[X_{k}].\end{align*}
By taking the maximum over $i$ from $1$ to $\infty$ for the left hand side of the above inequality, we arrive at $\max_{1\leq k\leq i}\mathbb{E}[X_{k}]\leq \frac{ \mathbb{E}[\varepsilon]}{1-\Vert\mathbb{E}[\boldsymbol{\xi}]\Vert_{1}}$. Since it holds for every $i$, we have $\sup_{k\in\mathbb{N}}\mathbb{E}[X_{k}]\leq \frac{ \mathbb{E}[\varepsilon]}{1-\Vert\mathbb{E}[\boldsymbol{\xi}]\Vert_{1}}$, which implies that $\frac{X_n}{n}\rightarrow 0$ in probability as $n\rightarrow \infty$ by Markov inequality.

Next, we bound the first term of $\epsilon_n$. We know that $\sum_{k=n-i-1}^{\infty}\mathbb{E}[\xi_k]$ decrease to $0$ when $n\rightarrow \infty$ by the assumption $\Vert\mathbb{E}[\boldsymbol{\xi}]\Vert_{1}<1$. So we have $\frac{\sum_{i=1}^{n-1}\sum_{k=n-i-1}^{\infty}\mathbb{E}[\xi_k]\mathbb{E}[X_i]}{n}\rightarrow 0$ as $n\rightarrow \infty$. That derives that $\frac{\sum_{i=1}^{n-1}\sum_{k=n-i-1}^{\infty}\mathbb{E}[\xi_k]X_i}{n} \rightarrow 0$ in probability as $n\rightarrow \infty$.
So we obtain that $\frac{\epsilon_n}{n}\rightarrow 0$ in probability as $n\rightarrow \infty$.

From the definition of $M_n$ and independence among $\xi_{l}^{(n,k)}$ and $\xi_{l}^{(n,\tilde{k})}$ for $k \neq \tilde{k}$ also $\xi_{k}$ and $\varepsilon_{n}$, we have that
\begin{equation}\label{martingale-difference-square}
\begin{aligned}
\mathbb{E}[M^2_n]&=\sum_{i=1}^n\mathbb{E}[(X_i-\mathbb{E}[X_i|\mathcal{F}_{i-1}])^2]=\sum_{i=1}^n\mathbb{E}\left[\left(\varepsilon_{i}+\sum_{k=1}^{i-1}\sum_{l=1}^{X_{i-k}}\xi_{l}^{(i,k)}
                 -\mathbb{E}[\varepsilon]-\sum_{k=1}^{i-1}\sum_{l=1}^{X_{i-k}}\mathbb{E}[\xi_k]\right)^2\right]\\
                 &=\sum_{i=1}^n\mathbb{E}[(\varepsilon_{i}-\mathbb{E}[\varepsilon])^2]
                 +\sum_{i=1}^n\mathbb{E}\left[\left(\sum_{k=1}^{i-1}\sum_{l=1}^{X_{i-k}}\xi_{l}^{(i,k)}-\mathbb{E}[\xi_k]\right)^2\right]\\
                 &=\sum_{i=1}^n{\rm Var}[\varepsilon]+\sum_{i=1}^n \sum_{k=1}^{i-1}\mathbb{E}\left[\left(\sum_{l=1}^{X_{i-k}}\xi_{l}^{(i,k)}-\mathbb{E}[\xi_k]\right)^2\right]=\sum_{i=1}^n{\rm Var}[\varepsilon]+\sum_{i=1}^n \sum_{k=1}^{i-1}{\rm Var}[\xi_k]\mathbb{E}[X_{i-k}]\\
                 &\le n{\rm Var}[\varepsilon]+\Vert{\rm Var}[\boldsymbol{\xi}]\Vert_{1}\sum_{i=1}^n \mathbb{E}[X_i].
\end{aligned}
\end{equation}
So we divide both sides of this inequality by $n^2$ with the fact $\frac{\sum_{i=1}^n\mathbb{E}[X_i]}{n}\rightarrow \frac{\mathbb{E}[\varepsilon]}{1-\Vert\mathbb{E}[\boldsymbol{\xi}]\Vert_{1}}$, we obtain that $\mathbb{E}[M^2_n/n^2]\rightarrow 0$ as $n\rightarrow \infty$, that derives that $\frac{M_n}{n}\rightarrow 0$ in probability as $n\rightarrow \infty$. Together with $\frac{\epsilon_n}{n}\rightarrow 0$ and (\ref{martingale decomposition}), we conclude that
$$\frac{1}{n}{\sum_{k=1}^{n}X_{k}}\rightarrow\frac{\mathbb{E}[\varepsilon]}{1-\Vert\mathbb{E}[\boldsymbol{\xi}]\Vert_{1}},$$
in probability as $n\rightarrow\infty$.

By using the LDP for $\mathbf{P}\left(\frac{1}{n}{\sum_{k=1}^{n}X_{k}}\in\cdot\right)$ in Theorem \ref{thm:LDP}, the tail probabilities have exponential decay, that are summable. That is, for any $\delta>0$, there exists some $N\in\mathbb{N}$, such that
for any $n>N$, we have
\begin{align*}
\mathbf{P}\left(\left|\frac{1}{n}{\sum_{k=1}^{n}X_{k}}-\frac{\mathbb{E}[\varepsilon]}{1-\Vert\mathbb{E}[\boldsymbol{\xi}]\Vert_{1}}\right|\ge \delta\right) 
&\leq
\mathbf{P}\left(\frac{1}{n}{\sum_{k=1}^{n}X_{k}}\ge \mu+\delta\right)
+\mathbf{P}\left(\frac{1}{n}{\sum_{k=1}^{n}X_{k}}\leq\mu- \delta\right)
\\
&\leq
e^{-\frac{n}{2}I(\mu+\delta)}+e^{-\frac{n}{2}I(\mu-\delta)},
\end{align*}
where $\mu=\frac{\mathbb{E}[\varepsilon]}{1-\Vert\mathbb{E}[\boldsymbol{\xi}]\Vert_{1}}$
and $I(\cdot)$ is the rate function from Theorem \ref{thm:LDP}, 
which implies that 
  $$\sum_{n=1}^{\infty}\mathbf{P}\left(\left|\frac{1}{n}{\sum_{k=1}^{n}X_{k}}-\frac{\mathbb{E}[\varepsilon]}{1-\Vert\mathbb{E}[\boldsymbol{\xi}]\Vert_{1}}\right|\ge \delta\right) 
  \leq 
  N+\sum_{n=N+1}^{\infty}\left(e^{-\frac{n}{2}I(\mu+\delta)}+e^{-\frac{n}{2}I(\mu-\delta)}\right)< \infty.$$
Therefore, by Borel-Cantelli lemma, it follows that
\begin{equation}
\label{eq:lln3}
\lim_{n\rightarrow\infty}\frac{1}{n}{\sum_{k=1}^{n}X_{k}}=\mu:=\frac{\mathbb{E}[\varepsilon]}{1-\Vert\mathbb{E}[\boldsymbol{\xi}]\Vert_{1}},
\end{equation}
almost surely as $n\rightarrow\infty$.
\end{proof}

\begin{proof}[Proof of Corollary~\ref{Coll1:LLN}]
When $\varepsilon$ is Poisson with mean $\alpha_{0}$, $\xi_{k}$ are Poisson with mean $\alpha_{k}$, $k=1,2,\ldots$,
we can compute
$\Vert\mathbb{E}[\boldsymbol{\xi}]\Vert_{1}$ which is given by $\Vert\mathbb{E}[\boldsymbol{\xi}]\Vert_{1}=\sum_{i=1}^{\infty}\mathbb{E}[{\xi}_{i}]=
    \sum_{i=1}^{\infty}\alpha_{i}=\Vert \alpha\Vert_{1}$. By applying Theorem \ref{thm:LLN} we have
\begin{align*}
\lim_{n\rightarrow\infty}\frac{1}{n}{\sum_{k=1}^{n}X_{k}}=\mu:=\frac{\alpha_0}{1-\Vert\alpha\Vert_{1}}.
\end{align*}
\end{proof}

\begin{proof}[Proof of Corollary~\ref{Coll2:LLN}]
Consider INAR(1) model where we assume $\xi_{k}\equiv 0$ for every $k\geq 2$. Then we can compute $\Vert\mathbb{E}[\boldsymbol{\xi}]\Vert_{1}=\sum_{i=1}^{\infty}\mathbb{E}[{\xi}_{i}]=
    \mathbb{E}[{\xi}_{1}]$. By applying Theorem~\ref{thm:LLN} we have
\begin{align*}
\lim_{n\rightarrow\infty}\frac{1}{n}{\sum_{k=1}^{n}X_{k}}=\mu:=\frac{\mathbb{E}[\varepsilon]}{1- \mathbb{E}[{\xi}_{1}]}.
\end{align*}	
\end{proof}

\subsection{Proof of Central limit theorem}\label{Proof of CLT}

\begin{proof}[Proof of Theorem~\ref{Thm:CLT}]
By the moment generating function method, it suffices to show that for any $\theta\in \mathbb{R}$
\begin{equation}\label{mgf convergence}
{\mathbb{E}\left[e^{\frac{\theta}{\sqrt{n}}\left(\sum\limits_{k=1}^{n}X_{k}-n\mu\right)}\right]}\rightarrow e^{\frac{\theta^2}{2}\sigma^2},
\end{equation}
as $n\rightarrow\infty$, where $$\sigma^2=\frac{\mathbb{E}[\varepsilon]\Vert{\rm Var}[\boldsymbol{\xi}]\Vert_{1}+{\rm Var}[\varepsilon](1-\Vert\mathbb{E}[\boldsymbol{\xi}]\Vert_{1})}{(1-\Vert\mathbb{E}[\boldsymbol{\xi}]\Vert_{1})^3}.$$
The proof is similar as the proof of MDP. It follows from equation (\ref{eq:char Nt1}) that
\begin{equation}\label{MGF2}
\begin{aligned}
\mathbb{E}\left[e^{\frac{\theta}{\sqrt{n}}\left(\sum\limits_{k=1}^{n}{X_{k}-n\mu}\right)}\right]
=&\frac{\prod_{i=1}^{n}\mathbb{E}[e^{\varepsilon G_n(k)}]}{e^{\sqrt{n}\theta\mu}}\\
=&e^{\sum_{i=1}^n\left(G_n(k)\mathbb{E}[\varepsilon]+\frac{1}{2}G_n^2(k){\rm Var}
[\varepsilon]\right)-\sqrt{n}\theta\mu+O\left(\frac{1}{n}\right)},
\end{aligned}
\end{equation}
where 
\begin{equation}\label{G_n3}
G_n(k)=\frac{\theta}{\sqrt{n}}\overline{G}_1(k)+\left(\frac{\theta}{\sqrt{n}}\right)^2\overline{G}_2(k)+\varepsilon_{n}(k),
\end{equation}
and $\overline{G}_1(k), \overline{G}_2(k)$ satisfy the following equations:
\begin{equation}\label{G_1(k)2}
\overline{G}_1(k)=1+\sum\limits_{i=1}^{k-1}\mathbb{E}[\xi_{i}]\overline{G}_1(k-i)\qquad \rm with \qquad  \overline{G}_1(1)=1,
\end{equation}
\begin{equation} \label{G_2(k)2}
\overline{G}_2(k)=\sum\limits_{i=1}^{k-1}\mathbb{E}[\xi_{i}]\overline{G}_2(k-i)
+\frac{1}{2}\sum\limits_{i=1}^{k-1}{\rm Var}[\xi_{i}]\overline{G}_1^2(k-i)\qquad \rm with  \qquad  \overline{G}_2(1)=0,
\end{equation}
and note that $\varepsilon_{n}(k)$ is the remainder term of the second order asymptotic expansion. 
So combine the equation (\ref{G_1(k)2}) and (\ref{G_2(k)2}) into equation (\ref{MGF2}) we arrive at 

\begin{equation}\label{mgf convergence1}
\begin{aligned}
&\mathbb{E}\left[e^{\frac{\theta}{\sqrt{n}}\left(\sum\limits_{k=1}^{n}{X_{k}-n\mu}\right)}\right]\\
=&e^{\left[\mathbb{E}[\varepsilon]\left(\frac{\theta}{\sqrt{n}}{\overline{G}_1(k)}
+\left(\frac{\theta}{\sqrt{n}}\right)^2{\overline{G}_2(k)}\right)
+\frac{{\rm Var}[\varepsilon]}{2}\left(\frac{\theta}{\sqrt{n}}{\overline{G}_1(k)}+\left(\frac{\theta}{\sqrt{n}}\right)^2{\overline{G}_2(k)}\right)^2\right]-\sqrt{n}\mu\theta+O(\frac{1}{n})}\\
=&e^{\frac{\mathbb{E}[\varepsilon]\theta}{\sqrt{n}}\left(\sum\limits_{k=1}^{n}{\overline{G}_1(k)}-
n\mu\right)+\frac{{\theta}^2}{n}\left(\mathbb{E}[\varepsilon]\sum\limits_{k=1}^{n}{\overline{G}_2(k)}+
\frac{{\rm Var}[\varepsilon]}{2}\sum\limits_{k=1}^{n}\overline{G}_1^2(k)\right)+O(\frac{1}{n})}.
\end{aligned}
\end{equation}
Under the Assumption \ref{CLT-Assumption}  (b2) we know from the counterpart of equation (\ref{sum control}) and (\ref{sum convergence}) that

\begin{equation}\label{sum control2}
\sum\limits_{i=n}^{\infty}{\mathbb{E}[\xi_{i}]}\leqslant\sum\limits_{i=n}^{\infty}{C_2{i}^{-a}}
<\frac{C_2}{-a+1}(n-1)^{-a+1}.
\end{equation}
Therefore, we conclude that
\begin{equation}\label{sum convergence2}
\mathbb{E}[\varepsilon]\theta\sqrt{n}\left\vert\frac{\sum\limits_{i=n}^{\infty}
{\mathbb{E}[\xi_{i}]}}{(1-\Vert\mathbb{E}[\boldsymbol{\xi}]\Vert_{1})^2}\right\vert <\frac{\vert\mathbb{E}[\varepsilon]
\vert\vert\theta\vert{C_2n^{\frac{3}{2}-a}}}{(a-1)(1-\Vert\mathbb{E}[\boldsymbol{\xi}]\Vert_{1})^2}\rightarrow0,
\end{equation}
as $n\rightarrow\infty$. That means the first part of the last equality in equation (\ref{mgf convergence1})
$$e^{\frac{\mathbb{E}[\varepsilon]\theta}{\sqrt{n}}\left(\sum\limits_{k=1}^{n}{\overline{G}_1(k)}-
n\mu\right)}\rightarrow 0,  ~~~~~~~~  {\rm as}~n\rightarrow\infty.$$
Associated with the two facts equations (\ref{sum G_1^2}) and (\ref{sum G_2}) we arrive at that 
for any $\theta\in \mathbb{R}$ 
\begin{equation}\label{mgf convergence2}
\begin{aligned}
\mathbb{E}\left[e^{\frac{\theta}{\sqrt{n}}\left(\sum\limits_{k=1}^{n}{X_{k}-n\mu}\right)}\right]
=&e^{\frac{\mathbb{E}[\varepsilon]\theta}{\sqrt{n}}\left(\sum\limits_{k=1}^{n}{\overline{G}_1(k)}-
n\mu\right)+\frac{{\theta}^2}{n}\left(\mathbb{E}[\varepsilon]\sum\limits_{k=1}^{n}{\overline{G}_2(k)}+
\frac{{\rm Var}[\varepsilon]}{2}\sum\limits_{k=1}^{n}\overline{G}_1^2(k)\right)+O(\frac{1}{n})}\\
\rightarrow & e^{\theta^2\left(\mathbb{E}[\varepsilon]\frac{\Vert{\rm Var}[\boldsymbol{\xi}]\Vert_{1}}{2(1-\Vert\mathbb{E}[\boldsymbol{\xi}]\Vert_{1})^3}+\frac{{\rm Var}[\varepsilon]}{2}\frac{1}{(1-\Vert\mathbb{E}[\boldsymbol{\xi}]\Vert_{1})^2}\right)},
\end{aligned}
\end{equation}
as $n\rightarrow\infty$. This provides the desired result for CLT and completes the proof.
\end{proof}

\begin{proof}[Proof of Corollary~\ref{Coll1:CLT}]
When $\varepsilon$ is Poisson with mean $\alpha_{0}$, $\xi_{k}$ are Poisson with mean $\alpha_{k}$, $k=1,2,\ldots$, then we can compute $\mathbb{E}[\varepsilon]={\rm Var}[\varepsilon]=\alpha_{0}, \Vert\mathbb{E}[\boldsymbol{\xi}]\Vert_{1}=\sum_{i=1}^{\infty}\mathbb{E}[{\xi}_{i}]=
    \sum_{i=1}^{\infty}\alpha_{i}=\Vert \alpha\Vert_{1}, \Vert{\rm Var}[\boldsymbol{\xi}]\Vert_{1}=\sum_{i=1}^{\infty}\rm Var[{\xi}_{i}]=\sum_{i=1}^{\infty}\alpha_{i}=\Vert \alpha\Vert_{1}$. By applying Theorem~\ref{Thm:CLT} we have
\begin{equation*}
\frac{\sum_{k=1}^{n}X_{k}-n\mu}{\sqrt{n}}\rightarrow N\left(0, \frac{\alpha_0}{(1-\Vert\alpha\Vert_{1})^3} \right)
\end{equation*}
in distribution as $n\rightarrow \infty$.
This completes the proof.	
\end{proof}

\begin{proof}[Proof of Corollary~\ref{Coll2:CLT}]
Consider INAR(1) model where we assume $\xi_{k}\equiv 0$ for every $k\geq 2$. Then we can compute
$\Vert\mathbb{E}[\boldsymbol{\xi}]\Vert_{1}=\sum_{i=1}^{\infty}\mathbb{E}[{\xi}_{i}]=
    \mathbb{E}[{\xi}_{1}]$ and 
$\Vert\rm Var[\boldsymbol{\xi}]\Vert_{1}=\sum_{i=1}^{\infty}\rm Var[{\xi}_{i}]=\rm Var[{\xi}_{1}]$.

By applying Theorem~\ref{Thm:CLT} we have
\begin{equation*}\frac{\sum_{k=1}^{n}X_{k}-n\mu}{\sqrt{n}}\rightarrow N\left(0, \frac{\mathbb{E}[\varepsilon]{\rm Var}[\xi_1]+{\rm Var}[\varepsilon](1-\mathbb{E}[\xi_1])}
{(1-\mathbb{E}[\xi_1])^3} \right)
\end{equation*}
in distribution as $n\rightarrow \infty$.
This completes the proof.	
\end{proof}

\subsection{Technical Lemmas}
\begin{lemma}\label{Lem1:MDP}
For any $n\in\mathbb{N}$, $\overline{G}_1(n)\leqslant\frac{1}{1-\Vert\mathbb{E}[\boldsymbol{\xi}]\Vert_{1}}$,
where $\overline{G}_1(k)=1+\sum\limits_{i=1}^{k-1}{\mathbb{E}[\xi_{i}]\overline{G}_1(k-i)}$ for $k\geqslant2$ and $\overline{G}_1(1)=1$.
\end{lemma}

\begin{proof}[Proof of Lemma~\ref{Lem1:MDP}]
We prove Lemma~\ref{Lem1:MDP} by induction on $k$. By assumption $\Vert\mathbb{E}[\boldsymbol{\xi}]\Vert_{1}<1$, then $\overline{G}_1(1)\leqslant\frac{1}{1-\Vert\mathbb{E}[\boldsymbol{\xi}]\Vert_{1}}$. Now let's assume that $\overline{G}_1(s)\leqslant\frac{1}{1-\Vert\mathbb{E}[\boldsymbol{\xi}]\Vert_{1}}$ for all $s\leqslant{k}$. Then we can compute,
\begin{align*}
\overline{G}_1(k+1)&=1+\sum\limits_{i=1}^{k}{\mathbb{E}[\xi_{i}]\overline{G}_1(k+1-i)}
\leqslant1+\sum\limits_{i=1}^{k}{\mathbb{E}[\xi_{i}]}\frac{1}{1-\Vert\mathbb{E}[\boldsymbol{\xi}]\Vert_{1}}
\\
&=\frac{1-\Vert\mathbb{E}[\boldsymbol{\xi}]\Vert_{1}+\sum\limits_{i=1}^{k}{\mathbb{E}[\xi_{i}]}}{1-\Vert\mathbb{E}[\boldsymbol{\xi}]\Vert_{1}}
\leqslant\frac{1}{1-\Vert\mathbb{E}[\boldsymbol{\xi}]\Vert_{1}}.
\end{align*}
Hence, we proved that for every $n\in\mathbb{N}$, $\overline{G}_1(n)\leqslant\frac{1}{1-\Vert\mathbb{E}[\boldsymbol{\xi}]\Vert_{1}}$.
\end{proof}

\begin{lemma}\label{Lem2:MDP}
For any $n\in\mathbb{N}$, $\overline{G}_2(n)\leqslant\frac{\Vert{\rm Var}[\boldsymbol{\xi}]\Vert_{1}}{2(1-\Vert\mathbb{E}[\boldsymbol{\xi}]\Vert_{1})^3}$,
where
\begin{align*}\overline{G}_2(n)=\sum\limits_{i=1}^{n-1}{\mathbb{E}[\xi_{i}]\overline{G}_2(n-i)}
+\frac{1}{2}\sum\limits_{i=1}^{n-1}{{\rm Var}[\xi_{i}]\overline{G}_1^2(n-i)}
,\end{align*}
for $n\geqslant2$ and $\overline{G}_2(1)=0$.
\end{lemma}
\begin{proof}[Proof of Lemma~\ref{Lem2:MDP}]
We prove Lemma~\ref{Lem2:MDP} by induction on $k$. By assumption $\Vert\mathbb{E}[\boldsymbol{\xi}]\Vert_{1}<1$, it is not hard to see $\overline{G}_2(1)\leqslant\frac{\Vert{\rm Var}[\boldsymbol{\xi}]\Vert_1}{2(1-\Vert\mathbb{E}[\boldsymbol{\xi}]\Vert_{1})^3}$. Now let's assume that $\overline{G}_2(k)\leqslant\frac{\Vert{\rm Var}[\boldsymbol{\xi}]\Vert_1}{2(1-\Vert\mathbb{E}[\boldsymbol{\xi}]\Vert_{1})^3}$. for all $k\leqslant{n}$, then we can compute,
\begin{align*}
\overline{G}_2(k+1)&=\sum\limits_{i=1}^{k}{\mathbb{E}[\xi_{i}]\overline{G}_2(k+1-i)}+\frac{1}{2}
\sum\limits_{i=1}^{k}{{\rm Var}[\xi_{i}]\overline{G}_1^2(k+1-i)}
\\
&\leqslant\sum\limits_{i=1}^{k}{\mathbb{E}[\xi_{i}]}\frac{\Vert{\rm Var}[\boldsymbol{\xi}]\Vert_{1}}{2(1-\Vert\mathbb{E}
[\boldsymbol{\xi}]\Vert_{1})^3}+\frac{1}{2}\sum\limits_{i=1}^{k}{{\rm Var}[\xi_{i}]}\frac{1}{(1-\Vert\mathbb{E}[\boldsymbol{\xi}]\Vert_{1})^2}
\\
&\leqslant\frac{\Vert\mathbb{E}[\boldsymbol{\xi}]\Vert_{1}\Vert{\rm Var}[\boldsymbol{\xi}]\Vert_{1}}{2(1-\Vert\mathbb{E}[\boldsymbol{\xi}]\Vert_{1})^3}+\frac
{\Vert{\rm Var}[\boldsymbol{\xi}]\Vert_{1}}{2(1-\Vert\mathbb{E}[\boldsymbol{\xi}]\Vert_{1})^2}
=\frac{\Vert{\rm Var}[\boldsymbol{\xi}]\Vert_{1}}{2(1-\Vert\mathbb{E}[\boldsymbol{\xi}]\Vert_{1})^3}.
\end{align*}
Hence, we proved that for every $n\in\mathbb{N}$, $\overline{G}_2(n)\leqslant\frac{\Vert{\rm Var}[\boldsymbol{\xi}]\Vert_{1}}{2(1-\Vert\mathbb{E}[\boldsymbol{\xi}]\Vert_{1})^3}$.
\end{proof}

\section*{Disclosure Statement}
The authors declare that there are no financial or non-financial competing interests related to this work.
\section*{Acknowledgments}
	Nian Yao was supported in part by the Natural Science Foundation of China under Grant 12071361,
the Natural Science Foundation of Guangdong Province under Grant 2020A1515010822
and Shenzhen Natural Science Fund (the Stable Support Plan Program 20220810152104001). 
We wish to thank Prof Lingjiong Zhu for his helpful discussions and guidance.






\bibliographystyle{alpha}
\bibliography{mybibfile}
\end{document}